\documentclass[runningheads]{llncs}
\usepackage{graphicx}
\usepackage[symbol]{footmisc}

\usepackage{float}
\usepackage{mathtools}
\usepackage{amssymb, amsmath, latexsym}
\usepackage{nicefrac}
\usepackage{hhline}
\usepackage{makecell}
\usepackage{caption}
\usepackage{multirow}
\usepackage{wrapfig}

\usepackage{colortbl}
\definecolor{bgcolor}{rgb}{0.8,1,1}
\definecolor{bgcolor2}{rgb}{0.68,0.78,0.81}
\usepackage{threeparttable}

\usepackage{pifont}
\usepackage[colorinlistoftodos,bordercolor=orange,backgroundcolor=orange!20,linecolor=orange,textsize=scriptsize]{todonotes}
\usepackage{algorithm}\usepackage{algpseudocode}

\allowdisplaybreaks
\graphicspath{ {images/} }
\def\<#1,#2>{\langle #1,#2\rangle}

\begin{document} 
\title{Real Acceleration of Communication Process in Distributed Algorithms with Compression\thanks{The research was supported by Russian Science Foundation (project No. 23-11-00229).}}
\titlerunning{Real Acceleration in Distributed Algorithms with Compression}
%
\author{Svetlana Tkachenko\inst{1}
Artem Andreev\inst{1}
Aleksandr Beznosikov\inst{1,2}
Alexander Gasnikov\inst{1,2}
}
\authorrunning{S. Tkachenko, A. Andreev, A. Beznosikov, A. Gasnikov}
%
\institute{Moscow Institute of Physics and Technology, Moscow, Russia
}
\institute{Institute for Information Transmission Problems, Moscow, Russia
}
\maketitle              
\begin{abstract}

Modern applied optimization problems become more and more complex every day. Due to this fact, distributed algorithms that can speed up the process of solving an optimization problem through parallelization are of great importance. The main bottleneck of distributed algorithms is communications, which can slow down the method dramatically. One way to solve this issue is to use compression of transmitted information. In the current literature on theoretical distributed optimization, it is generally accepted that as much as we compress information, so much we reduce communication time. But in reality, the communication time depends not only on the size of the transmitted information, but also, for example, on the message startup time. In this paper, we study distributed optimization algorithms under the assumption of a more complex and closer-to-reality dependence of transmission time on compression. In particular, we describe the real speedup achieved by compression, analyze how much it makes sense to compress information, and present an adaptive way to select the power of compression depending on unknown or changing parameters of the communication process.

\end{abstract}
\section{Introduction}

Modern realities pose more and more complex optimization problems that need to be solved. For example, to improve the generalization of deployed models, machine learning engineers need to rely on training datasets of ever increasing sizes and on elaborate large-scale over-parametrized models \cite{arora2018optimization}. Therefore, it is increasingly necessary to resort to the use of distributed approaches to solving the optimization problem. The essence of distributed optimization is to the process of streamlining the target function by using multiple computing resources that are scattered across different machines or servers. It enables optimization algorithms to run in parallel, which can greatly increase the speed and efficiency of finding the optimal solution. Therefore, distributed optimization is widely used in various domains, including machine learning, data science, and operations research \cite{verbraeken2020survey}. 


However, when utilizing parallel computation in a distributed optimization environment, a common challenge is the communication between the computational devices. Since the agents function independently, they must exchange information to harmonise their local solutions and revise the global solution. Meanwhile, communication time is a waste that prevents full parallelization. Therefore, to struggle for effective communication and to address the communication bottleneck issue is a key point in distributed optimization \cite{konevcny2016federated,smith2018cocoa,ghosh2020communication}.

Employing compression of forwarded information is one of the viable solutions to decrease communication expenses \cite{seide20141,alistarh2017qsgd}. It assists in reducing file size while preserving important information. With the use of effective compression algorithms, transmission time can be considerably reduced both in theory and in practice \cite{gorbunov2021marina}.

Several models describing the dependence of transmission time on message size can be found in literature. The most frequently utilized model in the theoretical optimization is $T = \beta s,$ where $T$ is the transmission time, $\beta$ is the delay-size relationship, and $s$ is the message size. Meanwhile, there is a more practical and widespread model that has stayed away from theoretical optimization. This model is $T = \beta s + \alpha,$ where $\alpha$ is the server initialization time \cite{chan2007collective}. The simpler model indicate that transmission time can be reduced by a factor of $n$ by transmitting $n$ times less information. Nevertheless, the practical results contrast with the theoretical ones. In actuality, messaging involves initializing the channel, which refers to establishing a connection between the sender and the recipient. The second model accounts for this. This implies that there is minimal distinction when transmitting 1 or 2 bits, but once we send 100 Mb and 200 Mb, the variance is substantial. Therefore, we necessitate an accurate representation to characterise the communications.

Considering the issue of communication expenses and the proposed solution, the main questions of this study can be posed:

\begin{quote}
    \textit{1. Which model better describes the real world of messaging?}

    \textit{2. How does this change the theory of distributed optimization?}

    \textit{3. How can we determine the parameters of this model?} 

    \textit{4. What is the most efficient method of calculating the parameters of this model in the event of frequent data updates?}
\end{quote}

\subsection{Contributions}

\hspace{0.35cm} \textbf{More practical communication model:} Instead of the classical delay versus size model of $T=\beta s$ (where $\beta$ represents the relationship between delay and size), we have adopted a more realistic approach of $T=\beta s + \alpha$ (taking into account $\alpha$ -- the server initialization time). When a worker sends a message to the server, the channel initialization time contributes significantly to the small size of message transmission. It is, therefore, essential to consider this factor.

\textbf{Impact of model on communication complexities:} We analyze how the more practical model from the previous paragraph affects the communication costs of modern distributed algorithms with compression. We consider state-of-the-art methods that have the best theoretical guarantees for convex and non-convex problems.

\textbf{Estimate of $\alpha$, $\beta$:} In order to calculate the two coefficients $\alpha$ and $\beta$ based on the real data on the dependence of delay on message volume, we assume that $\alpha = \alpha_{const} + \delta_\alpha$ and $\beta = \beta_{const} + \delta_\alpha$ where $\delta_\alpha$ and $\delta_\beta$ follow independent normal distributions. $\alpha_{const}$ and $\beta_{const}$ are the true values of these coefficients, $\delta \alpha$ and $\delta \beta$ are errors of measurement or calculation.

Using this information, it is possible to calculate the coefficients $\alpha_{const}$ and $\beta_{const}$ through statistical techniques, such as least squares. Rather than storing the complete dataset of message size and delay time, we can update the summation of variables. This approach enables us to update just four variables without needing to recalculate the coefficients using the least squares method.

The estimation process aims to find the values of $\alpha$ and $\beta$ that best fit the observed data, thus providing insight into the server initialization time and the delay-volume relationship. 




\section{Problem statement}

We consider the optimization problems of the form: 
$$
\min \limits_{x \in \mathbb{R}^d} \left\{ f(x) := \frac{1}{n} \sum \limits_{i=1}^n f_i(x) \right\},
$$ where $x$ is the optimization variable. For example, in the context of ML, $x \in \mathbb{R}^d$ contains the parameters of the statistical model to be trained, $n$ -- number of employees/devices and functions $f_i(x) : \mathbb{R}^d \to \mathbb{R}$ -- model data loss $x$, stored on the device $i$.

\subsection{Distributed optimization with compression}

We give an illustration of the traditional use of distributed optimization, utilizing the gradient descent algorithm as an instance -- see Algorithm \ref{alg:appr1}.

\begin{algorithm}[h!]
   \caption{}
   \label{alg:appr1}
\begin{algorithmic}[1]
\State {\bf Initialization:} choose $x^0\in \mathbb{R}^d$ and stepsizes $\{\gamma_k\}_{k=0}^{K}$
\For{$k = 0, 1, \dots, K$}
    \State Server sends $x^k$ to all $n$ nodes
    \State Each $i$-th node, in parallel with the others, calculates the gradient of its corresponding function $f_i$:$$\nabla f_i(x^k)$$
    \State All nodes send $\nabla f_i(x^k)$ to the server
    \State Server performs aggregation:$$x^{k+1}=x^k- \gamma_k \cdot \dfrac{1}{n}\sum\limits_{i=1}^n\nabla f_i(x^k)$$
\EndFor  
\end{algorithmic}
\end{algorithm}

As noted above to handle large data sets, compression is necessary. In Algorithm \ref{alg:appr1} this can be represented by the compression operator $\mathcal{C} : \mathbb{R}^d \rightarrow \mathbb{R}^d $. In particular, $\nabla f_i(x^k)$ in line 4 should be replaced by $\mathcal{C}(\nabla f_i(x^k))$ and, accordingly, in line 6 we aggregate the compressed gradients $\mathcal{C}(\nabla f_i(x^k))$:
$$
x^{k+1}=x^k- \gamma_k \cdot \dfrac{1}{n}\sum\limits_{i=1}^n \mathcal{C}(\nabla f_i(x^k)).
$$
This approach is basic, but does not give the best convergence results \cite{gorbunov2020unified,beznosikov2020biased}. Once can note that more advanced methods with compression use more tricky schemes, in particular, they are based on various variance reduction techniques, which prescribe to compress not the gradient itself, but the difference between the gradient and some reference value \cite{mishchenko2019distributed,li2020acceleration,gorbunov2021marina,EF21,beznosikov2022distributed,beznosikov2022compression,beznosikov2023similarity}.








The theory of convergence of methods with compression is based on a formal definition of the properties of $\mathcal{C}$ operators. In particular, two classes of operators: unbiased and biased, are often distinguished in the literature.

\begin{definition} 
    $\mathcal{C}$ is an unbiased compression with $\zeta \geq 1$ if $\mathcal{C}$ is unbiased \\ ($\mathbb{E}[\mathcal{C}(x)] = x$) and $\mathbb{E}\left[\|\mathcal{C}(x)\|^2_2\right]\leq\zeta\|x\|^2_2$ for all $x\in\mathbb{R}^d$.
\end{definition}

\begin{definition}
    $\mathcal{C}$ is a biased compression with $\delta \geq 1$ if \\  $\mathbb{E}\left[\|\mathcal{C}(x)-x\|^2_2\right]\leq\left(1-1/\delta\right)\|x\|^2_2$ for all $x\in\mathbb{R}^d$.
\end{definition}

Meanwhile, these definitions do not give a complete picture about compression operators. The definitions are interesting for proving convergence and obtaining iterative complexity of algorithms. But to obtain the communication cost in the amount of transmitted information, it is necessary to understand how much the operator reduces the transmitted information.


\subsection{Degree of compression}

In this subsection, we estimate the degree of compression $\omega_{inf} = \tfrac{\textnormal{len}(x)}{\textnormal{len}(\mathcal{C}(x))}$, where $\textnormal{len}(x)$ is the number of bits of information to send $x \in \mathbb{R}^d$. We consider different classical compression operators.



\begin{definition}
    For $k \in [d] := \{1,\dots,d\}$, the {unbiased random  (aka Rand-$k$) sparsification} operator is defined via
    \begin{equation*}
        \mathcal{C}(x) := \frac{d}{k}\sum \limits_{i\in S}x_ie_i,
    \end{equation*}
    where $S\subseteq [d]$ is the $k$-nice sampling; i.e., a subset of $[d]$ of cardinality $k$ chosen uniformly at random, and $e_1,\dots,e_d$ are the standard unit basis vectors in $\mathbb{R}^d$.
\end{definition}

\begin{lemma}
    For the unbiased random sparsification $\omega_{inf} = \tfrac{d}{k}$.
\end{lemma}

\begin{proof}
    Initial vector $x$ contains $d$ non-zero coordinates, and the compressed one contains $k$. Then    $\omega_{inf} = \tfrac{d}{k}$. Here it is important to clarify that in the general case it is necessary to forward numbers of non-zero coordinates as well. But if the same random generator with the same seed is installed on the sending and receiving devices, it is possible to synchronize the randomness for free, and then there is no need to send additional information. 
\end{proof}
 



    
 



\begin{definition}[see \cite{alistarh2018convergence}]
    Top-$k$ sparsification operator is defined via
    \begin{equation*}
        \mathcal{C}(x) := \sum \limits_{i=d-k+1}^d x_{(i)} e_{(i)},
    \end{equation*}
where coordinates are ordered by their magnitudes so that $|x_{(1)}| \leq |x_{(2)}| \leq \cdots \leq |x_{(d)}|$.
\end{definition}
Top-$k$ is a greedy version of unbiased random sparsification.
\begin{lemma}
    For Top-$k$ sparsification $\omega_{inf} = \tfrac{d \cdot \text{len}(x) }{k \cdot \text{len}(x) + k \cdot \lceil \log_2 d \rceil}$.
\end{lemma}

\begin{proof}
    Similar to Unbiased random sparsification initial vector $x$ contains $d$ non-zero coordinates, and the compressed one contains $k$. But here, unlike random sparsification, we have to pass the numbers of selected non-zero coordinates. To encode the numbers from $1$ to $d$, $\lceil \log_2 d \rceil$ bits are needed. The total number of transmitted bits is $k \cdot \text{len}(x) + k \cdot \lceil \log_2 d \rceil$.  Then $\omega_{inf} = \tfrac{d \cdot \text{len}(x) }{k \cdot \text{len}(x) + k \cdot \lceil \log_2 d \rceil}$.
\end{proof}


\begin{definition}[see \cite{horvath2022natural}]
    Natural compression operator $\mathcal{C}_{nat}$ is defined as follows:
 \begin{equation*}
\mathcal{C}(x)= \begin{cases} \text{sign} (x) \cdot 2^{\lfloor{\log_2 |x|}\rfloor} , \text{ with }  p(x),\\
\text{sign} (x) \cdot  2^{\lceil{\log_2 |x|}\rceil} , \text{ with } 1-p(x),
\end{cases}
\label{def:rr_vec}
 \end{equation*}
 where probability $ p(x) := \frac{2^{\lceil{\log_2 |x|}\rceil}-|x|}{2^{\lfloor{\log_2 |x|}\rceil}}.$
\end{definition}
The essence of this compression is random rounding to the nearest power of two. In terms of computing on a computer with 32bit float type, this is simply equivalent to using only the sign bit and 8 bits from the exponent.
\begin{lemma}
    For Natural compression $\omega_{inf} = \tfrac{32}{9}$.
\end{lemma}
\begin{proof}
The statement follows directly from the use of such compression with 32bit float. Instead of 32 bits we send 9.
\end{proof}

\begin{definition}[see \cite{vogels2019powersgd}]
    Rank-r Power compression introduced by \cite{vogels2019powersgd} is a compressed-decompressed approach approach based on the low-rank approximate decomposition of the matrix $X \in \mathbb{R}^{n \times m}$ (transformed version of the original parcel vector $x$).
    

\end{definition}

\begin{lemma}
    For Rank-r PowerSGD compression $\omega_{inf} = \tfrac{nm}{r(n+m)}$.
\end{lemma}

\begin{proof}
    The product of matrices $PQ^T, P\in\mathbb{R}^{n\times r}, Q\in\mathbb{R}^{m\times r}$ approximates the matrix $X\in\mathbb{R}^{n\times m}$, Thus, instead of storing $n\cdot m$ numbers must be $r\cdot n+r\cdot m$. Then $\omega_{inf} = \tfrac{nm}{r(n+m)}$.
\end{proof}
The results obtained above are summarized in Table \ref{tab:1}.

\vspace{-0.5cm}
\renewcommand{\arraystretch}{2.5}
\renewcommand{\tabcolsep}{6pt}
\begin{table}[h!]
\centering
\begin{tabular}{c|c} \hline
    \cellcolor{bgcolor2}{{{\bf Compression operator }}} & \cellcolor{bgcolor2}{{ $\omega_{inf}$}}  \\ \hline
    Unbiased random sparsification & $\dfrac{d}{k}$ 
    \\ \hline
    Top-$k$ sparsification \cite{alistarh2018convergence}& $\frac{d \cdot \text{len}(x) }{k \cdot \text{len}(x) + k \cdot \lceil \log_2 d \rceil}$
    \\ \hline
    Natural compression \cite{horvath2022natural}& $\dfrac{32}{9}$
    \\ \hline
     Rank-$r$ Power compression \cite{vogels2019powersgd}& $\dfrac{nm}{r(n+m)}$ 
     \\ \hline
\end{tabular}
\vspace{0.3cm}
\caption{$\omega_{inf}$ for different compression operators.}
\label{tab:1}
\end{table}
\vspace{-1.cm}
As previously stated, the estimation of communication time necessitates $\omega_{inf}$. This coefficient serves as a reliable indicator of the effectiveness of the compression operator in each instance, particularly if one knows the optimal frequency of message compression (which is the objective of this research paper).

\section{Main part}

\subsection{Transmission time model and convergence complexities}

\begin{wrapfigure}[16]{r}{8cm}
     \vspace{-1cm}
     \centering
    \includegraphics[width=0.5\textwidth]{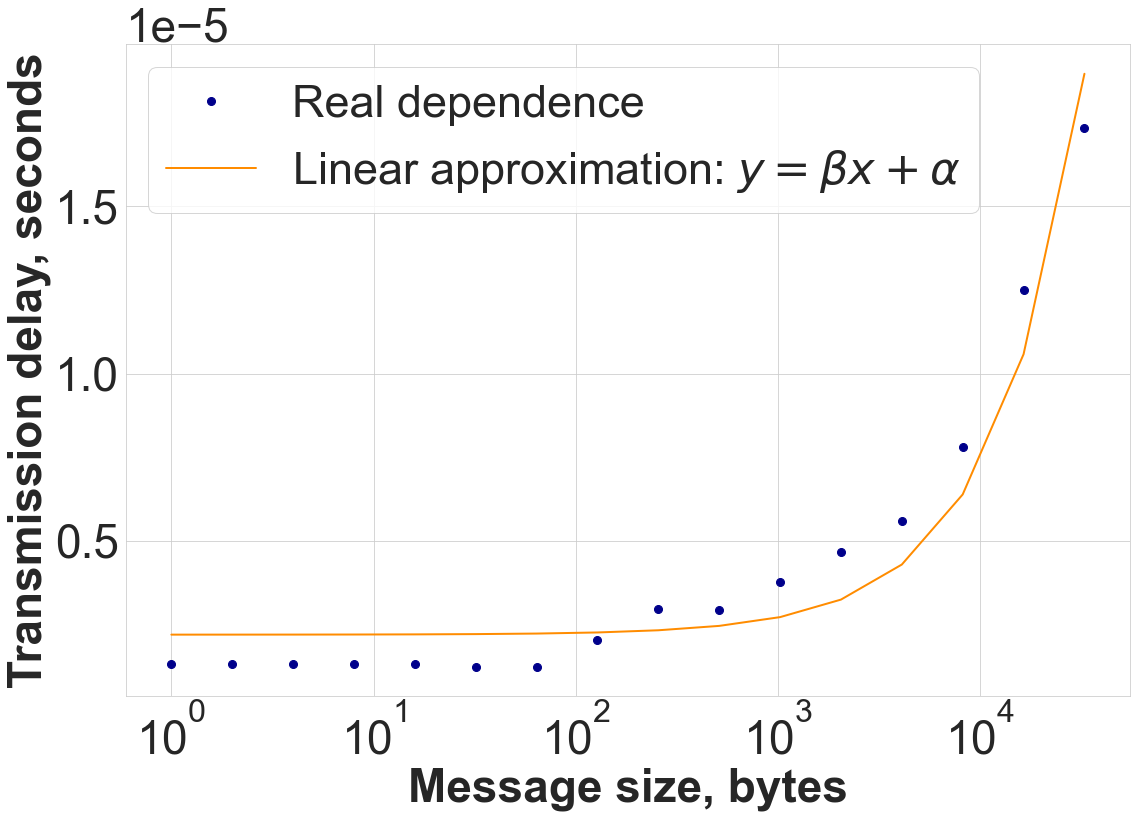}
    \caption{Dependence of communication time on the size of the transmitted messages for MSU supercomputer "Lomonosov": blue dots -- real values, orange line -- approximation.}
    \label{Approx}
\end{wrapfigure}
We consider the following model of transmission time:
$$
T(s) = \alpha + \beta \cdot s,$$
where $s$ is the size of the packages, $\beta$ represents the time to transmit one unit of the information, $\alpha$ represents the time to initialize the channel, which is the delay that occurs before any message transmission occurs. This delay may involve activities such as creating a connection, verifying the user's identity, or loading essential resources. As mentioned earlier, in papers on theoretical optimization and convergence estimates, it is assumed that $\alpha = 0$. But let us examine the plot presented in Figure~\ref{Approx}, which portrays the relationship between the transmission delay and message size in a live network. Using this plot we can see the effect of $\alpha$ on communication time.

Note that the theoretical results on the communication cost of distributed algorithms with compression depend on the parameter 
\begin{equation}
    \label{eq:beta}
    \eta = \frac{T(\text{len}(\nabla f_i (x)))}{T(\text{len}(\nabla f_i (x)) / w_{inf})}. 
\end{equation}
In particular, the best results for method with unbiased compression for convex \cite{li2020acceleration} and non-convex \cite{gorbunov2021marina} problems linear depends on $\left(\tfrac{1}{\eta} + \tfrac{\zeta}{\eta \sqrt{n}}\right)$. The state-of-the-art results for biased compression \cite{qian2021error,EF21} linear depends on $\left(\tfrac{1}{\eta} + \tfrac{\delta}{\eta}\right)$. It is easy to see that if $\alpha = 0$, the expression \eqref{eq:beta} gives $\eta = w_{inf}$. But if $\alpha \neq 0$, it is possible that $\eta = 1$, thus the impact of even a large $w_{inf}$ can be almost completely canceled. 





\subsection{Division into areas}

\begin{wrapfigure}[14]{r}{8cm}
 \vspace{-0.8cm}
\centering
\includegraphics[width=0.5\textwidth]{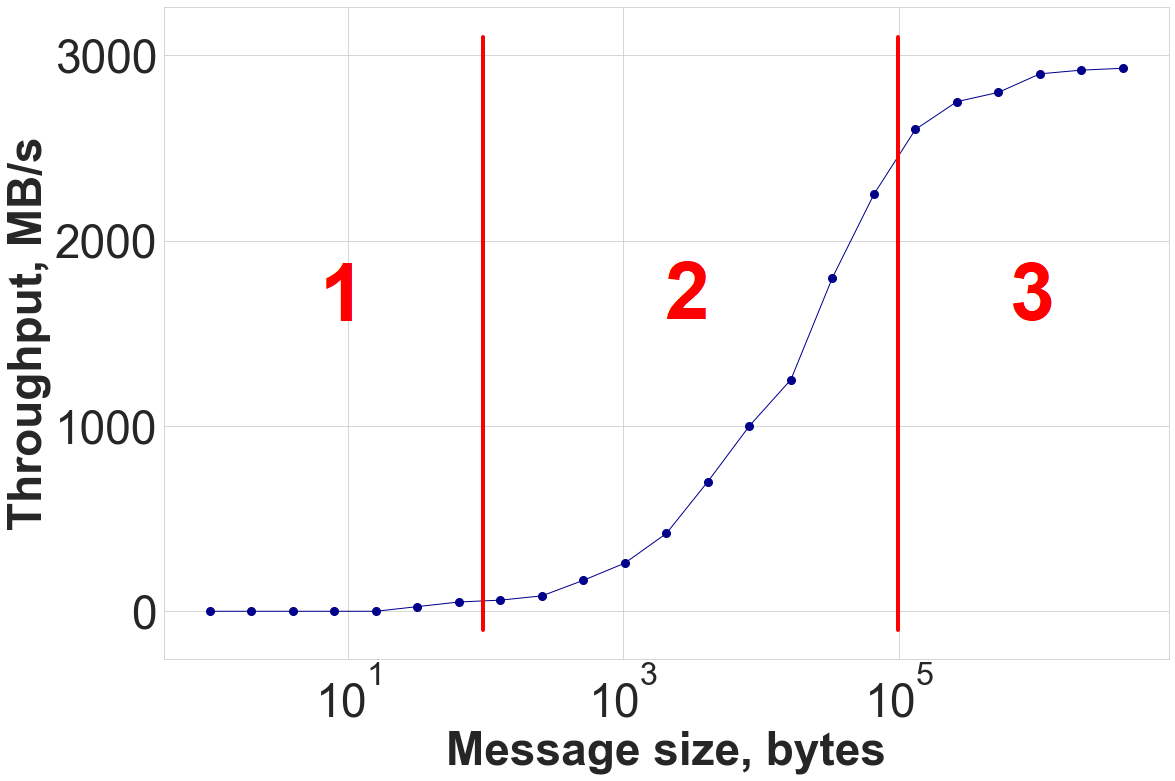}
    \caption{Division into fields according to the importance of the summands $\alpha$ and $\beta s$}
    \label{Areas}
\end{wrapfigure}

Let us examine the original plot (Figure~\ref{Approx}) and divide it into three conditional ranges (Figure~\ref{Areas}). In the first range, the coefficient $\alpha$ is the most significant. This means that if the size of the message $s$ falls within this range, then $\alpha$ greatly exceeds $\beta s$. In the second range, both coefficients $\alpha$ and $\beta$ hold value, with $\alpha$ and $\beta s$ being close in value. In the third range, $\beta$ carries the most significance, with $\beta s$ greatly exceeding $\alpha$.

Let us examine how the transmission delay varies with changes in message size (see Figure~\ref{p212}, Figure~\ref{p323}, and Figure~\ref{p31}). We determine the number of times the message size alters during compression, and subsequently how many times the transmission time changes:

$\bullet$ When transitioning between areas 3 to 3 and 3 to 2, the message is compressed by a factor of $n$, while communication time is reduced by $0,95 \cdot n$. 

$\bullet$ When transitioning between areas 2 to 2 and 2 to 1, compression is approximately 40 times greater than the reduction in communication cost. 

\begin{figure}[h!]
\begin{center}
\begin{minipage}[h]{0.32\linewidth}
    \includegraphics[width=1\textwidth]{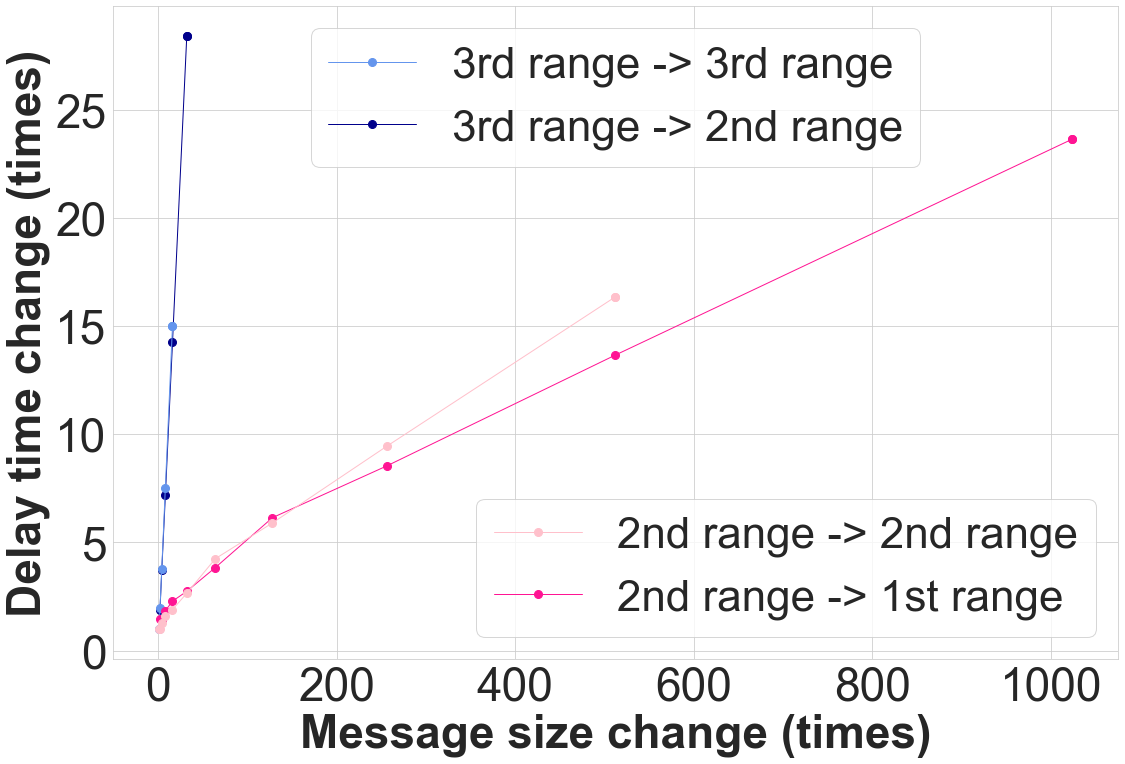}
    \caption{Transitions from area 2 to areas 1, 2}
    \label{p212}
\end{minipage}
\hfill 
\begin{minipage}[h]{0.32\linewidth}
\includegraphics[width=1\textwidth]{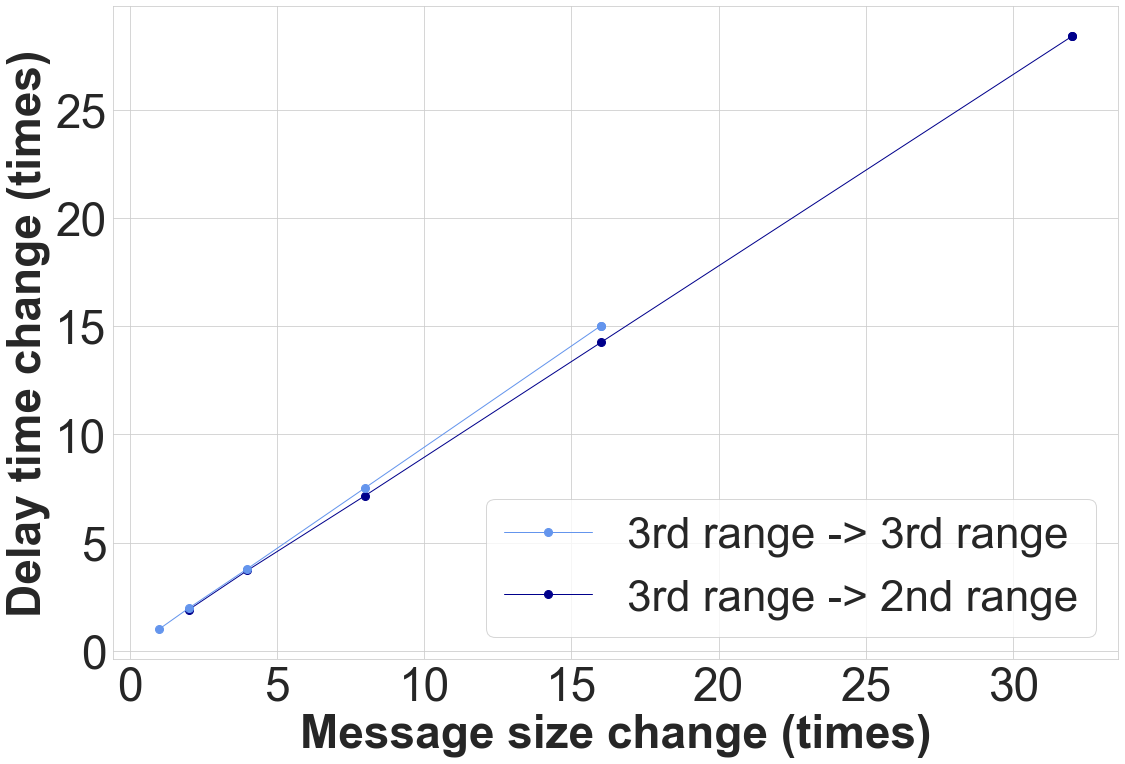}
    \caption{Transitions from area 3 to areas 2, 3}
    \label{p323}
\end{minipage}
\hfill 
\begin{minipage}[h]{0.32\linewidth}
\includegraphics[width=1\textwidth]{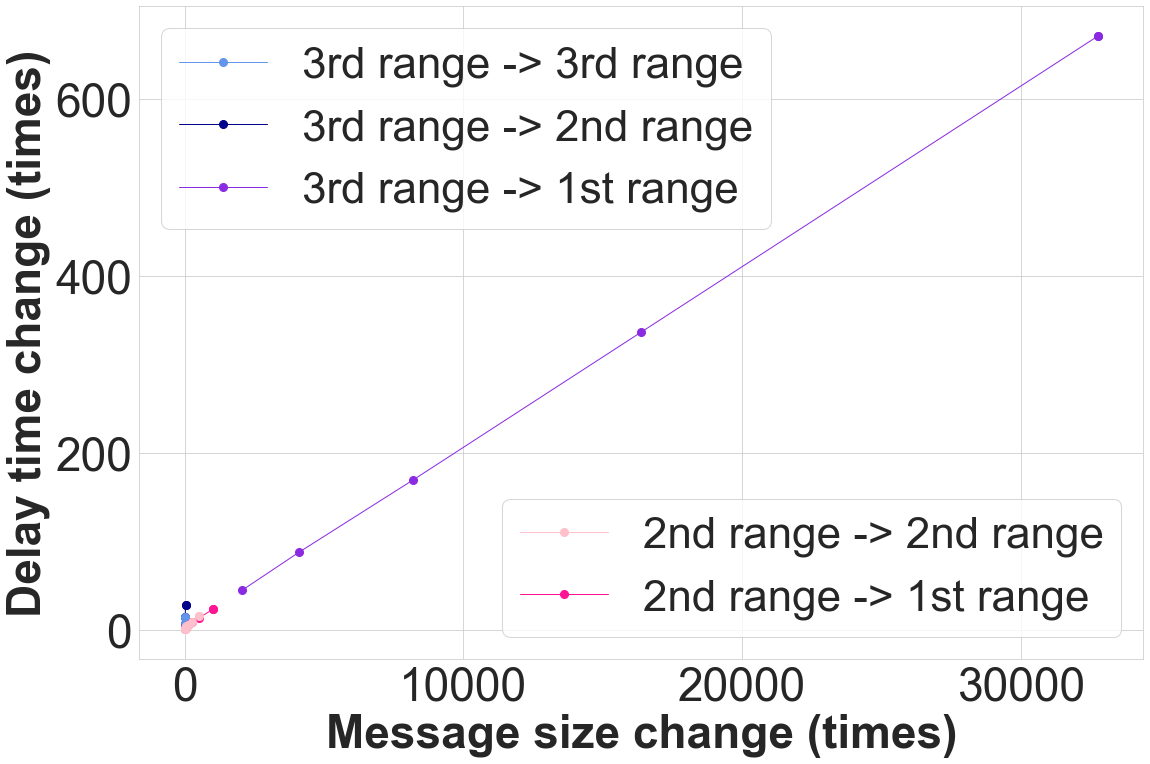}
    \caption{Transitions from area 3 to area 1}
    \label{p31}
\end{minipage}
\end{center}
\vspace{-1cm}
\end{figure}

$\bullet$ When transitioning between areas 3 to 1, the time saved is insignificant compared to the compression size (with a compression of 5000 times, the transmission time is only reduced by approximately 100 times).

\textbf{Conclusion:} It is most feasible to travel from area 3 to 3 or 2, and it is not financially viable to travel from area 3 to 1.

\begin{wrapfigure}[12]{r}{8cm}
    \vspace{-0.7cm}
    \centering
    \includegraphics[width=0.5\textwidth]{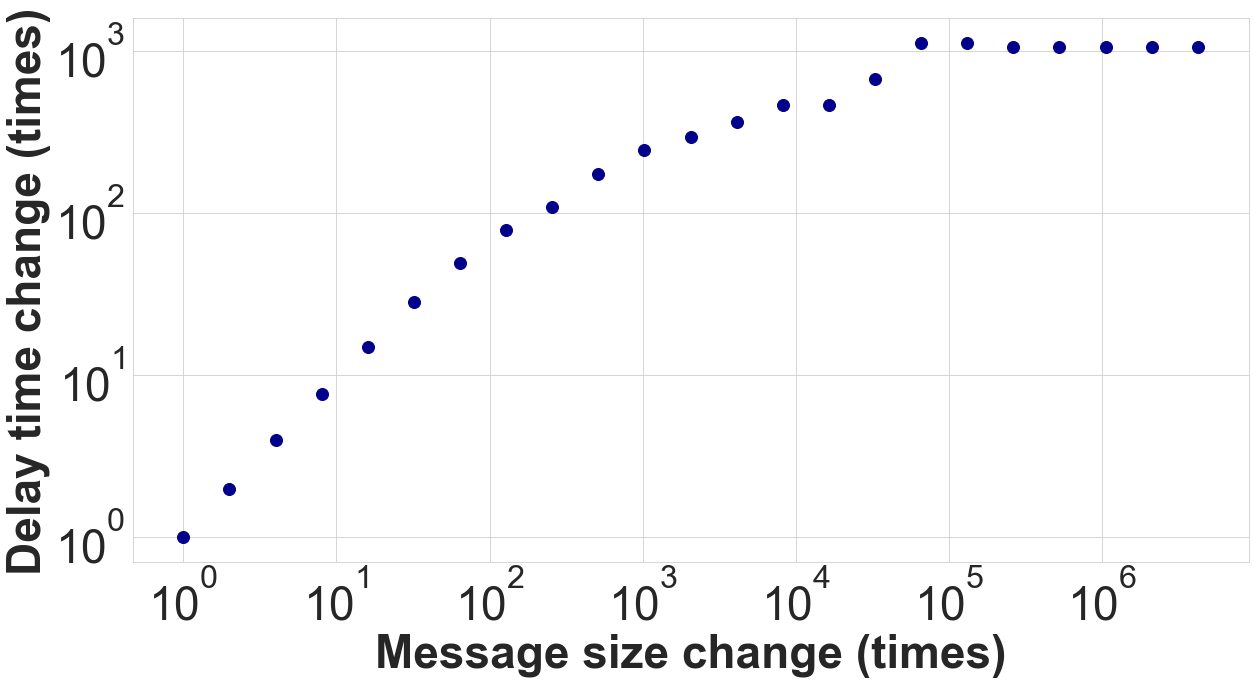}
    \caption{Transmission delay reduction on message size reduction}
    \label{comp}
\end{wrapfigure}

Let us consider an example how compression of a message affects transmission time. We use the distributed system from Figure~\ref{Approx}. The size of an uncompressed message is $10^{22}$. Figure \ref{comp}  demonstrates that the variation in time is almost linear at first, but then the compression loses its effectiveness. Hence, it can be concluded that the compression of a message has an impact on the transmission time.

\textbf{Conclusion:}  High compression does not result in a significant time savings, thus extensively compressing a message is not an efficient approach.

\subsection{A way to find $\alpha$ and $\beta$ for an unknown network parameters}

Here is a method for determining $\alpha$ and $\beta$ when the network parameters are unknown, when we can specify areas as in Figure \ref{Areas}. We formalize the problem of finding or estimating $\alpha$ and $\beta$ as follows.
 \vspace{0.25cm}
 
\textbf{Condition:} It is possible to transmit messages of varying sizes ranging from 0 to $P_{max}$, which represents the maximum message size that we can send. It is imperative to consider technological constraints when evaluating the feasibility of message transmission. For instance, for each message, the values $\alpha$ and $\beta$ are stochastic and have the laws: $\alpha(t) = \alpha_{const} + \delta \alpha$ and $\beta(t) = \beta_{const} + \delta \beta$. That is, $\alpha(t)$ and $\beta(t)$ vary among samples and consist of a constant value plus stochastic noise $\delta$ (which follows a normal distribution with mean 0 and variance $\sigma^2$, where $\sigma = \alpha_m \cdot \alpha_{const}$ or $\sigma = \beta_m \cdot \beta_{const}$ depending on the nature of $\delta$). \vspace{0.25cm}

\textbf{Suggested solution:} Let us apply the formulas of the method of least squares to recalculate $\alpha$ and $\beta$. During the operation of the main optimization algorithm we vary message sizes. Firstly, we calculate the delay at 2 points. 
Then, for each subsequent step, we select (determistically or randomly) the next point from within the interval $(0, P_{max})$ and calculate the delay. Here is an example of code that executes this algorithm:
\vspace{-0.5cm}
\begin{algorithm}[h!]
   \caption{}
   \label{alg:appr}
\begin{algorithmic}[1]
\State {\bf Parameters:} largest message size $P_{max}$;
\State {\bf Initialization:} compute times $y_1, y_2$ for message sizes $x_1, x_2$ respectively ($y = \beta \cdot x + \alpha$). Set $s_x^2 = x_1 + x_2$, $s_y^2 = y_1 + y_2$, $s_{xy}^2 = y_1\cdot x_1 + y_2\cdot x_2$, $s_{xx}^2 = x_1\cdot x_1 + x_2\cdot x_2$;

\For{$k = 3, 4, \dots$}
    \State for new $x_k \in [0, P_{max}]$, compute $y_k$
    \State $s_x^k = s_x^{k-1} + x_k$
    \State $s_y^k = s_y^{k-1} + y_k$
    \State $s_{xy}^k = s_{xy}^{k-1} + x_k\cdot y_k$
    \State $s_{xx}^k = s_{xx}^{k-1} + x_k\cdot x_k$
    \State $\beta_k = \frac{k \cdot s_{xy}^k - s_x^k \cdot s_y^k}{k \cdot s_{xx}^k - (s_x^k)^2}$
    \State $\alpha_k = \frac{s_y^k - \beta \cdot s_x^k}{k}$
\EndFor  
\end{algorithmic}
\end{algorithm}
\vspace{-0.5cm}

The algorithm recalculates the $\alpha$ and $\beta$ coefficients using the least squares formulas. 
It is worth pointing out that it is very expensive to recalculate the parametrs $\alpha$ and $\beta$ using the least squares method and to store all data of message sizes and times of transmission $\{x_i, y_i \}$. But Algorithm \ref{alg:appr} can works online. We need only 4 variables: the sum of message sizes $s_x$, the sum of delays $s_y$, and the sums needed for the least squares calculation $s_{xy}$ and $s_{xx}$.


\begin{proposition}
    $\beta_k, \alpha_k$ from Algorithm \ref{alg:appr} are unbiased estimations of $\beta$ and $\alpha$, namely $\Bbb{E}[\beta_k] = \beta_{const}$ and $\Bbb{E}[\alpha_k] = \alpha_{const}$.
\end{proposition}
\begin{proof}
    We start from $\beta_k$:
    \begin{align}
    \label{eq:1}
        \Bbb{E}[\beta_k] 
        = 
        \Bbb{E}\left[\dfrac{k \cdot s_{xy}^k - s_x^k \cdot s_y^k}{k \cdot s_{xx}^k - (s_x^k)^2}\right]
        = \dfrac{k\cdot\Bbb{E}[s_{xy}^k] - s_x^k\cdot\Bbb{E}[s_y^k]}{k \cdot s_{xx}^k - (s_x^k)^2}.
    \end{align}
    Next, we estimate $\Bbb{E}[s_{xy}^k] $ and $s_x^k\cdot\Bbb{E}[s_{y}^k]$
    \begin{align}
    \label{eq:2}
        \Bbb{E}[s_{xy}^k] =& \Bbb{E}\left[\sum\limits_{i=1}^k(x_iy_i)\right] = \left(\sum\limits_{i=1}^k(x_i \Bbb{E}[y_i])\right) = \sum\limits_{i=1}^kx_i(\beta_{const}x_i+\alpha_{const})
        \notag\\
        = & \beta_{const}\sum\limits_{i=1}^kx_i^2 + \alpha_{const}\sum\limits_{i=1}^kx_i,
    \end{align}
    \begin{align}
    \label{eq:3}
        s_x^k\cdot\Bbb{E}[s_{y}^k] 
        =& \left(\sum\limits_{i=1}^k x_i\right)\cdot\Bbb{E}\left[\sum\limits_{i=1}^ky_i\right] = \left(\sum\limits_{i=1}^k x_i\right)\cdot \sum\limits_{i=1}^k(\beta_{const}x_i + \alpha_{const}) = 
        \notag\\
        =&\beta_{const}\left(\sum\limits_{i=1}^kx_i\right)^2 + k\cdot\alpha_{const}\sum\limits_{i=1}^k x_i.
    \end{align}
    Substituting \eqref{eq:3} and \eqref{eq:2} to \eqref{eq:1}, we get
    $$
    \Bbb{E}[\beta_k] = \dfrac{k\cdot\left(\beta_{const}\cdot s_{xx}^k + \alpha_{const}\cdot s_x^k\right) - \beta_{const}\cdot (s_x^k)^2 - k\cdot\alpha_{const}\cdot s_x^k}{k \cdot s_{xx}^k - (s_x^k)^2} = \beta_{const}.
    $$
    Finally, for $\Bbb{E}[\alpha_k]$ we obtain
    \begin{align*}
        \Bbb{E}[\alpha_k] =& \dfrac{\Bbb{E}[s_y^k - \beta_k\cdot s_x^k]}{k} = 
        \dfrac{\Bbb{E}[s_y^k] - \Bbb{E}[\beta_k]\cdot s_x^k}{k}
        \\
        =& \dfrac{\sum\limits_{i=1}^k(\beta_{const}x_i + \alpha_{const}) - \beta_{const}\cdot s_x^k}{k} 
= \dfrac{k\cdot\alpha_{const}}{k} = \alpha_{const}.
    \end{align*}
\end{proof}

\section{Conclusions}

In this paper, we considered a realistic communication cost model $T(s) =\beta s + \alpha$, which takes into account $\alpha$ -- the server initialization time. We tried to discuss how it affects to communication time complexities of algorithms. We also provided the algorithm for determining the coefficients $\alpha$ and $\beta$ utilizing statistical techniques such as the least squares method, alongside estimated uncertainties related to this approach. Rather than storing the complete message size and delay time sets, it is viable to update some combinations of the variables.




One can note that the model considered in this paper can also be improved. For example, we can also include the time required for the communication operator counting. In some cases, this can be quite expensive, which slows down the computational process. Taking this time into account is an important detail for future research.

\bibliographystyle{splncs04}
\bibliography{ltr}



\end{document}